\documentclass[12pt]{article}
\usepackage{amsthm}
\usepackage{amsmath}
\usepackage{amssymb}
\usepackage{latexsym}
\usepackage{amsfonts}
\usepackage{color}
\usepackage{mathrsfs}
\usepackage{epsfig}
\usepackage{cite}
\newtheorem{theorem}{\bf Theorem}[section]

\newtheorem{lemma}[theorem]{\bf Lemma}

\begin{document}
\title{Diophantine Equation with Balancing-like Sequences Associated to the Pillai-Tijdeman-type Problem}
\author{Bijan Kumar Patel and Prashant Tiwari}
\date{}
\maketitle
\begin{abstract} \noindent
Let $\{x_{n}\}_{n \geq 0}$ be the balancing-like sequence defined by $x_{n+1} = A x_{n} - x_{n-1}$, for $A>2$, where $x_0 = 0$ and $x_1 = 1$. In this paper, we demonstrate how to find all the solutions of the Diophantine equation, $C_{1}x_{n_{1}} + C_{2}x_{n_{2}} + C_{3}x_{n_{3}} = C_{4}x_{n_4} + C_{5}x_{n_5} + C_{6}x_{n_{6}}$, in fixed integer $A \geq 3$, $n_1 > n_2 > n_3\geq 0, n_4 >n_5 > n_6 \geq 0,$ and $C_{1}x_{n_{1}} \neq C_{4} x_{n_4}$, where $C_{1}, C_{2}, C_{3}, C_{4}, C_{5}, C_{6}$ are given integers such that $C_{1} C_{2} C_{3} \neq 0$.
\end{abstract}

\section{Introduction}
A balancing number $B$ is a natural number satisfying the Diophantine
equation 
\[
\sum_{j=1}^{B-1} j = \sum_{k = B+1}^{R} k,
\]
for some natural number $R$ \cite{Behera}. If $B$ is a balancing number, then $8B^2 + 1$ is a perfect square. The $n$-th balancing number is denoted by $B_{n}$ and satisfy the binary recurrence $B_{n+1} = 6 B_{n} - B_{n-1}$ with initial values $(B_{0}, B_{1}) = (0, 1).$ The balancing sequence has been studied extensively and generalized in many ways \cite{Panda, Panda1}. As a generalization of the balancing sequence, Panda and Rout \cite{Panda} introduced a family of binary recurrences defined by
\[
x_{n+1} = A x_{n} - x_{n-1}, ~~x_0 = 0,~ x_1 = 1,
\]
for any fixed integer $A>2$. Afterwards, these sequences were called as balancing-like sequences $\{ x_n \}_{n \geq 0}$ since the particular case corresponding to A = 6 coincides with the balancing sequence. When $A = 2$, the binary sequence $x_{n+1} = 2 x_{n} - x_{n-1}$ with $x_0 = 0, x_1 = 1$ generates the natural numbers. The sequence $\{ x_n \}_{n \geq 0}$ behaves like the sequence of natural numbers, and for this reason, Khan and Kwong \cite{Khan} called these sequences as generalized natural number sequences. For each $n$, $D x_{n}^2 + 1$ is a perfect rational square, where $D = \frac{A^2 -4}{4}$, is a perfect square and its positive square root is called a Lucas-balancing-like number \cite{Panda}. They have also noted that all the balancing-like sequences are strong divisibility sequences. Further, different associate sequences of a balancing-like sequence has been studied by Panda and Pradhan \cite{Panda1}. The Binet form of the balancing-like sequence $\{x_{n} \}_{n \geq 0}$ is given by 
\[
x_n = \frac{\gamma^n - \delta^n}{\sqrt{A^2 - 4}}, ~~n = 1, 2, \dots,
\]
where $\gamma := \frac{A + \sqrt{A^2 - 4}}{2}$ and $\delta := \frac{A - \sqrt{A^2 - 4}}{2}$ are the roots of the characteristic equation $X^2 - AX + 1 = 0$. It is easy to see that $\gamma  = \delta^{-1}.$ Balancing-like sequence satisfies the following divisibility property.

\begin{lemma}[Theorem 2.8, \cite{Panda}]\label{lm}
For all $m, n \in \mathbb{N}$, $\gcd (x_m, x_n) = x_{\gcd(m, n)}$.
\end{lemma}
Furthermore, one can prove by induction that the inequality
\begin{equation}\label{eq1}
  \gamma^{n-2} \leq x_{n} \leq \gamma^{n-1}  
\end{equation}

holds for all positive integers $n$.\\

\noindent
Sahukar and Panda \cite{Sahukar} dealt with the Brocard-Ramanujan-type equations $x_{n_1} x_{n_2} \dots x_{n_k} \pm 1 = x_{m}$ or $y_m$ or $y_{m}^2$ where $\{x_{n}\}_{n \geq 0}$ and $\{y_{m}\}_{m \geq 0}$ are either balancing-like sequences or associated balancing-like sequences. Subsequently they \cite{Sahukar1} showed that each balancing-like sequence has at most three terms which are one away from perfect squares.

Shorey and Tijdeman \cite{Shorey} proved under suitable conditions, that the equation $Ax^m + By^m = C x^n + D y^n$ implies that max$\{n,m\}$ is bounded by a computable constant depending only on $A, B, C, D$, where $A, B, C, D, m, n > 0$. Recently, Ddamulira et al. \cite{Ddamulira} studied a variation of the equation $Ax^m + By^m = C x^n + D y^n$ with the terms of the Lucas sequence $\{U_{n}\}_{n \geq 0}$. They showed that there are two types of solutions of the Diophantine equation $AU_{n} + BU_{m} = CU_{n_1} + DU_{m_1}$, sporadic solutions and parametric solutions. 

Motivated by the above works, we study the Diophantine equation 
\begin{equation}\label{eq2}
    C_{1}x_{n_{1}} + C_{2}x_{n_{2}} + C_{3}x_{n_{3}} = C_{4}x_{n_4} + C_{5}x_{n_5} + C_{6}x_{n_{6}} 
\end{equation}
with $n_1 > n_2 > n_3\geq 0, n_4 >n_5 > n_6 \geq 0,$ and $C_{1}x_{n_{1}} \neq C_{4} x_{n_4}$, where $\{x_{n} \}_{n \geq 0}$ is a balancing-like sequence and $C_{1}, C_{2}, C_{3}, C_{4}, C_{5}, C_{6}$ are given integers such that $C_{1} C_{2} C_{3} \neq 0$. More precisely, our main result is the following.

\begin{theorem}\label{th:1}
Let $X := \max \{ |C_{1}|,|C_{2}|, |C_{3}|, |C_{4}|, |C_{5}|, |C_{6}|\}$ and $\psi:=(3+\sqrt{5})/2$ be the smallest possible $\gamma$.  Relabeling the variables $(n_1, n_2, n_3, n_4, n_5, n_6)$ as $(m_1, m_2, m_3, m_4, m_5, m_6),$ where $m_1 \geq m_2 \geq m_3 \geq m_4 \geq m_5 \geq m_6$. If $n_1 =n_4$, we rewrite the Diophantine equation \eqref{eq2} as $$(C_{1}-C_{4})x_{n_{1}} + C_{2}x_{n_{2}} + C_{3}x_{n_{3}} = C_{5}x_{n_5} + C_{6}x_{n_{6}},$$ and change $(C_1, C_2, C_3 ,C_4, C_5, C_6)$ to $(C_1-C_4, C_2, C_3, C_5, C_6, 0)$. Thus, $n_1>n_4$.   Furthermore, we change the sign of some of the coefficients $(C_1, C_2, C_3 ,C_4, C_5, C_6)$  so that the Diophantine equation \eqref{eq2} becomes 
\begin{equation}\label{e1}
    A_1 x_{m_{1}} + A_2 x_{m_{2}} + A_3 x_{m_{3}} + A_4 x_{m_{4}} + A_5 x_{m_{5}}+ A_6 x_{m_{6}}=0.
\end{equation}

\noindent
Assume $A \leq 308 X$ for any fixed integer $A >2$ of the characteristic equation $X^2 - AX + 1 = 0$. Then, the solutions of the Diophantine equation \eqref{e1} are of two types:
\begin{enumerate}
    \item[(i)] Sporadic ones. These are finitely many and they satisfy: 
    \begin{align*}
&m_6 \leq \frac{\log (12 X)}{\log \psi}, m_5 \leq \frac{\log (1000 X^3)}{\log \psi}, m_4 \leq \frac{\log (123000 X^5)}{\log \psi}, \\
&m_3 \leq \frac{\log (17700000 X^7)}{\log \psi}, m_2 \leq \frac{\log (4530000000 X^9)}{\log \psi}, m_1 \leq \frac{\log (99660000000 X^{10})}{\log \psi}.
\end{align*}

\item[(ii)] Parametric ones. These are of one of the two forms:
\[
(m_1, m_2, m_3, m_4, m_5, m_6) = (m_5 + l, m_5 + k, m_5 + j, m_5 + i, m_5, 0),
\]
where 
\begin{align*}
    &l \leq \frac{\log (104000000 X^7)}{\log \psi},~~~ k \leq \frac{\log (4720000 X^6)}{\log \psi},\\
    &j \leq \frac{\log (18500 X^4)}{\log \psi} ~~~\text{and}~~~i \leq \frac{\log (130 X^2)}{\log \psi},
\end{align*}
and $\gamma$ is a root of $A_{1}X^l + A_{2}X^k + A_{3}X^j + A_4X^i + A_5=0$, or of the form
\[
(m_1, m_2, m_3, m_4, m_5, m_6) = (m_6 + m, m_6 + l, m_6 + k, m_6 + j, m_6 + i, m_6), 
\]
  where 
  \begin{align*}
    m & \leq \frac{\log (8305000000 X^9)}{\log \psi},~~~ l \leq \frac{\log (377500000 X^8)}{\log \psi},~~~ k \leq \frac{\log (1485000 X^6)}{\log \psi},\\
    &j \leq \frac{\log (10300 X^4)}{\log \psi} ~~~\text{and}~~~i \leq \frac{\log (80 X^2)}{\log \psi},
\end{align*} 
and $\gamma$ is a root of 
$$
A_{1}X^m + A_{2}X^l + A_{3}X^k + A_{4}X^j + A_5 X^i + A_6=0.
$$

\end{enumerate}

\end{theorem}
\section{Main Result}
To achieve the objective of this paper, we need the following result.
\begin{lemma}
Let $C_{1}, C_{2}, C_{3}, C_{4}, C_{5}, C_{6}$ are given integers such that $C_{1} C_{2} C_{3} \neq 0$. If the Diophantine equation $C_{1}x_{n_{1}} + C_{2}x_{n_{2}} + C_{3}x_{n_{3}} = C_{4}x_{n_4} + C_{5}x_{n_5} + C_{6}x_{n_{6}}$ with $n_1 > n_2 > n_3\geq 0, n_4 >n_5 > n_6 \geq 0, C_{1} x_{n_{1}}\neq C_{4} x_{n_{4}}$ holds, then $A<308X$.
\end{lemma}

\begin{proof}
{\bf Case I} If $C_3 = C_6 =0$, then $C_{1}x_{n_{1}} + C_{2}x_{n_{2}} = C_{4}x_{n_4} + C_{5}x_{n_{5}} $ for any $n_{3}, n_{6} \geq 0$. We have the following subcases for it. 

{\it Subcase 1} If $C_{4}= C_{5} =0$, then $C_1 x_{n_{1}} = - C_2 x_{n_{2}}$ for any $n_4>n_5\geq 0$. As we have given $ C_{1} x_{n_{1}}\neq C_{4} x_{n_{4}}$ and $C_{1} C_{2}\neq 0$, using these one can get $n_{1}$ and  $n_{2}$ both are non-zero. Thus from Lemma \ref{lm} and $n_{2}\neq 0$, we have $\frac{x_{n_{1}}}{x_d}$ divides $C_2$, where $d:=\gcd(n_1,n_2).$ Write $n_1 =: ld$, where $l\geq 2$. If $d=1$, then $\frac{x_{n_{1}}}{x_d}=\frac{x_l}{x_1}=x_l\geq x_2=A$, so $A\leq X$. For $d\geq 2$, we have $\frac{x_{n_{1}}}{x_d}= \frac{\gamma^{ld}-\delta^{ld}}{\gamma^{d}-\delta^{d}}$. Now, we will show that 
\[
\frac{\gamma^{ld}-\delta^{ld}}{\gamma^{d}-\delta^{d}} > \gamma,
\]
which is equivalent to $\gamma^{ld} - \gamma^{d+1}> \delta^{ld}- \delta^{d-1}$. Since $d\geq 2$ and $|\delta^{d-1}| < 1$, therefore it is sufficient to show that $\gamma^{2d}-\gamma^{d+1} > 29$. The left–hand side is 
\begin{align*}
    \gamma^{2d}-\gamma^{d+1} &= \gamma^{d+1} (\gamma^{d-1}-1) \\
&\geq \gamma^{d+1} (\gamma-1).
\end{align*}
The smallest possible $\gamma$ is $\psi = \frac{3+\sqrt{5}}{2}$ (for $A=3$) and $\psi^{d+1} (\psi-1)\geq \psi^{3}(\psi-1) >29$. Hence, $\gamma < \frac{x_{ld}}{x_d} \leq X$, which gives $A = \gamma + \delta < \gamma < X$. Using the inequality (\ref{eq1}) we have $x_{n_{1}}\geq \gamma^{n_{1}-2}$ and $x_{d}\leq \gamma^{d-1}$, so $\frac{x_{n_{1}}}{x_{d}} \geq \frac{\gamma^{n_{1}-2}}{\gamma^{d-1}} = \gamma^{n_{1}-d-1}$. Since $d$ is a proper divisor of $n_{1}$ and $d\geq 2$, so $d\leq \frac{n_{1}}{2}$, therefore we have $\frac{x_{n_{1}}}{x_{d}} \geq \gamma^{\frac{n_{1}}{2}-1}$. Since $\frac{x_{n_{1}}}{x_{d}}$ divides $C_2$, we get 
\[
\gamma^{\frac{n_{1}}{2}-1}\leq |C_2|\leq X.
\]
Taking the logarithm of both sides of the above equation leads to
\begin{equation*}
   n_{1}\leq 2+ 2 \frac{\log X}{\log \gamma}.
\end{equation*}
Since $\gamma \geq \psi$, we have 
\begin{equation}\label{eqn3}
   0<n_{2}<n_{1}\leq 2+ 2 \frac{\log X}{\log \psi}.
\end{equation}
{\it Subcase 2} If one of $C_4, C_5$ is non-zero and the other is zero. Therefore, we can assume that $C_4\neq 0$ and $n_4 \neq 0$. Thus, if $C_5=0$, then 
\[
C_{1}x_{n_{1}} + C_{2}x_{n_{2}} = C_{4}x_{n_4}.
\]
If $n_4=n_1$, then $(C_{1}-C_{4})x_{n_{1}} = - C_{2}x_{n_{2}}  $, so from {\it subcase 1}, we get the bound $A \leq X$. If $n_4\neq n_1$ and $n_2=0$, then again from {\it subcase 1}, we get the bound $A \leq X$. If $n_4\neq n_1$ and $n_2 \neq 0$, then $C_{1}x_{n_{1}} + C_{2}x_{n_{2}} = C_{4}x_{n_4} $. In this situation we replace $(C_1, C_2, C_4 ,C_5) \longrightarrow (C_1, C_2, C_4, 0)$ and in this case also we get the same bound which is $A\leq X$. 

{\it Subcase 3} If both $C_4$ and $C_5$ are non-zero, then 
\[
C_{1}x_{n_{1}} + C_{2}x_{n_{2}} = C_{4}x_{n_4} + C_{5}x_{n_5}.
\] 
If $n_1=n_4$, then $(C_{1}-C_{4})x_{n_{1}} + C_{2}x_{n_{2}} = C_{5}x_{n_5}$. If $n_2=0$ or $n_5=0$, then again from {\it subcase 1}, we get the same bound $A\leq X$, otherwise we replace $(C_1, C_2, C_4 ,C_5) \longrightarrow (C_1-C_4, C_2, C_5, 0)$. The only effect is that $X$ is replaced with $2X$. If $n_1\neq n_4$, then switch $C_1$ with $C_4$. If needed, assume $n_1 = \max\{n_4, n_1\}$, then $n_1 > n_4$.
We relabel our indices $(n_1, n_2, n_4, n_5)$ as $(m_1, m_2, m_3, m_4)$ where $m_1 > m_2 \geq m_3 \geq m_4$, and the coefficients $C_1, C_2, C_4 ,C_5$ as $A_1, A_2, A_3, A_4$ and change signs to at most a couple of them so that our equation is now
\begin{equation}\label{eqn4}
A_1 x_{m_{1}} + A_2 x_{m_{2}} + A_3 x_{m_{3}} + A_4 x_{m_{4}}=0.
\end{equation}
Using Binet formula of $\{x_{n} \}_{n \geq 0}$, we have 
$$
|A_1| \gamma^{m_{1}} = |- A_2 (\gamma^{m_{2}}- \delta^{m_{2}}) - A_3 (\gamma^{m_{3}} -\delta^{m_{3}}) - A_4 (\gamma^{m_{4}} - \delta^{m_{4}}) - A_1 \delta^{m_{1}}| < 7X \gamma^{m_{2}}, 
$$ 
therefore 
\begin{equation}\label{eqn5}
    \gamma^{m_{1}-m_{2}} < 7X.
\end{equation}
Thus, since $m_1 > m_2$, we get that $ A < \gamma \leq \gamma^{m_{1}-m_{2}} < 7X$. Recalling that we might have to replace $X$ with $2X$, we get the final bound in this case $A < 14X$. \\

\noindent{}
{\bf Case II}  If one of $C_5, C_6$ is non-zero and the other is zero. We assume that $C_5\neq 0$ and $n_5 \neq 0$. Thus, if $C_6=0$, then 
\[
C_{1}x_{n_{1}} + C_{2}x_{n_{2}} + C_{3}x_{n_{3}} = C_{4}x_{n_4} + C_{5}x_{n_5}.
\]
If $n_1=n_4$, then $(C_{1}- C_{4})x_{n_{1}} + C_{2}x_{n_{2}} + C_{3}x_{n_{3}} = C_{5}x_{n_5}$. So using the {\it Case I}, we get the bound $A<14X$. Recalling that we might have to replace $X$ with $2X$, we get the final bound in this case $A < 28X$. If $n_1\neq n_4$, then switch $C_1$ with $C_4$. If needed, we may assume that $n_1 = \max\{n_4, n_1\}$, therefore $n_1 > n_4$. We relabel our indices $(n_1, n_2, n_3, n_4, n_5)$ as $(m_1, m_2, m_3, m_4, m_5)$ where $m_1 > m_2 > m_3 \geq m_4 \geq m_5$, and the coefficients $C_1, C_2, C_3 ,C_4, C_5$ as $A_1, A_2, A_3, A_4, A_5$ and change signs to at most a couple of them so that our equation is now
\begin{equation}\label{eqn6}
A_1 x_{m_{1}} + A_2 x_{m_{2}} + A_3 x_{m_{3}} + A_4 x_{m_{4}} + A_5 x_{m_{5}} =0.
\end{equation}
Using Binet formula of $\{x_{n} \}_{n \geq 0}$, we have 
\begin{align*}
    |A_1| \gamma^{m_{1}} &= |- A_2 (\gamma^{m_{2}}- \delta^{m_{2}}) - A_3 (\gamma^{m_{3}} -\delta^{m_{3}}) - A_4 (\gamma^{m_{4}} - \delta^{m_{4}}) \\
    &~~~~~~~~~~~~~~~~~~~~~~~~~~~~~~~~~~~- A_5 (\gamma^{m_{5}} - \delta^{m_{5}}) - A_1 \delta^{n_{1}}| < 9X \gamma^{m_{2}}, 
\end{align*}
therefore 
\begin{equation}\label{eqn7}
    \gamma^{m_{1}-m_{2}} < 9X.
\end{equation}
Thus, since $m_1 > m_2$, we get that $ A < \gamma \leq \gamma^{m_{1}-m_{2}} < 9X$. Recalling that we might have to replace $X$ with $14X$, we get the final bound in this case $A < 126X$. \\

\noindent{}
{\bf Case III} If both $C_5 , C_6$ are non-zero. In this case we have 
\[
C_{1}x_{n_{1}} + C_{2}x_{n_{2}} + C_{3}x_{n_{3}} = C_{4}x_{n_4} + C_{5}x_{n_5} + C_{6}x_{n_{6}}.
\]
If $n_4=n_1$, then $(C_{1}-C_{4})x_{n_{1}} + C_{2}x_{n_{2}} + C_{3}x_{n_{3}} = C_{5}x_{n_5} + C_{6}x_{n_{6}}$. If $n_2=0$ or $n_3=0$ or $ n_5=0$ or $ n_6=0$, then in all the situations we will use the first case and get the bound $A\leq 28X$, otherwise we replace $(C_1, C_2, C_3 ,C_4, C_5, C_6) \longrightarrow (C_1-C_4, C_2, C_3, C_5, C_6, 0)$. The only effect is that $28X$ is replaced with $56X$.  If $n_1\neq n_4$, then switch $C_1$ with $C_4$. if needed, we may assume that $n_1 = \max\{n_4, n_1\}$, therefore $n_1 > n_4$. We relabel our indices $(n_1, n_2, n_3, n_4, n_5, n_6)$ as $(m_1, m_2, m_3, m_4, m_5, m_6)$ where $m_1 > m_2 > m_3 \geq m_4 \geq m_5 \geq m_6$, and the coefficients $C_1, C_2, C_3 ,C_4, C_5, C_6$ as $A_1, A_2, A_3, A_4, A_5, A_6$ and change signs to at most a couple of them so that our equation is now
\begin{equation}\label{eqn8}
A_1 x_{m_{1}} + A_2 x_{m_{2}} + A_3 x_{m_{3}} + A_4 x_{m_{4}} + A_5 x_{m_{5}}+ A_6 x_{m_{6}}=0.
\end{equation}
Using Binet formula of $\{x_{n} \}_{n \geq 0}$, we have
\begin{align*}
    |A_1| \gamma^{m_{1}} &= |- A_2 (\gamma^{m_{2}}- \delta^{m_{2}}) - A_3 (\gamma^{m_{3}} -\delta^{m_{3}}) - A_4 (\gamma^{m_{4}} - \delta^{m_{4}}) \\
    &~~~~~~~~~~~~~~~~~~~~~~~~~~~~~~~~~~~- A_5 (\gamma^{m_{5}} - \delta^{m_{5}}) - A_6 (\gamma^{m_{6}} - \delta^{m_{6}}) - A_1 \delta^{n_{1}}| < 11X \gamma^{m_{2}}, 
\end{align*}
therefore 
\begin{equation}\label{eqn9}
    \gamma^{m_{1}-m_{2}} < 11X.
\end{equation}
Thus, since $m_1 > m_2$, we get that $ A < \gamma \leq \gamma^{m_{1}-m_{2}} < 11X$. Recalling that we might have to replace $X$ with $28X$, we get the desired conclusion which is $A<308X$.
\end{proof}

\begin{center}
    \bf{Proof of Theorem \ref{th:1}}
\end{center}
From the above lemma we get $A$ is bounded. It is possible for small $A$ that the equation has infinitely many solutions. But we will show that this is not the case, it has finitely many solutions. Using the substitution $(C_1, C_2, C_3 ,C_4, C_5, C_6) \longrightarrow (C_1-C_4, C_2, C_3, C_5, C_6, 0)$, and relabelling some of the variables, we may assume that $m_1 > m_2 > m_3 \geq m_4 \geq m_5 \geq m_6$ and that equation \eqref{eqn8} holds. Then estimate \eqref{eqn9} holds, So by using \eqref{eqn9}, we have
\begin{equation}\label{eqn10}
  m_{1} - m_{2} < \frac{\log (11X)}{\log \psi}.
\end{equation}
Using the Binet formula of balancing-like sequence in \eqref{eqn8}, we get
\begin{align}\label{eqn11}
   &\left\rvert \gamma^{m_2} \left( A_{1} \gamma^{m_1 - m_2} + A_{2} \right) - \left( \frac{A_1}{\gamma^{m_1}} + \frac{A_2}{\gamma^{m_2}} \right) \right\lvert \nonumber \\ 
   &~~~~~~~~~~~~~~~~~~~~~~~~~~= \left\rvert -A_{3} (\gamma^{m_3} - \delta^{m_3}) -A_{4} (\gamma^{m_4} - \delta^{m_4}) -A_{5} (\gamma^{m_5} - \delta^{m_5}) -A_{6} (\gamma^{m_6} - \delta^{m_6}) \right\lvert \nonumber \\
   &~~~~~~~~~~~~~~~~~~~~~~~~~~\leq 8X \gamma^{m_{3}}.
\end{align}
Since $m_{1} - m_{2} > 0$, $A_{1} \gamma^{m_1 - m_2} + A_{2} \neq 0$. Thus 
$
\left\rvert A_{1} \gamma^{m_1 - m_2} + A_{2}\right\lvert \left\rvert A_{1} \delta^{m_1 - m_2} + A_{2}\right\lvert \geq 1.
$
Since $\left\rvert A_{1} \delta^{m_1 - m_2} + A_{2}\right\lvert \leq 2X$, we get $\left\rvert A_{1} \gamma^{m_1 - m_2} + A_{2}\right\lvert \geq 1/ 2X$. Further
\[
\left\rvert \frac{A_1}{\gamma^{m_1}} + \frac{A_2}{\gamma^{m_2}} \right\lvert \leq \frac{2X}{\gamma^{m_2}}.
\]
Hence,
\[
\left\rvert \gamma^{m_2} \left( A_{1} \gamma^{m_1 - m_2} + A_{2} \right) - \left( \frac{A_1}{\gamma^{m_1}} + \frac{A_2}{\gamma^{m_2}} \right) \right\lvert \geq \frac{\gamma^{m_2}}{2X} - \frac{2X}{\gamma^{m_{2}}}.
\]
Now, assume 
\begin{equation}\label{eqn12}
    \frac{\gamma^{m_2}}{2X} - \frac{2X}{\gamma^{m_{2}}} \leq \frac{\gamma^{m_2}}{8X}.
\end{equation}
Then $\gamma^{2m_2} < 6X^2$, so $\gamma^{m_2} < 3X$. Hence
\begin{equation}\label{eqn13}
  m_{6} \leq m_{5} \leq m_{4} \leq m_{3} \leq  m_{2} \leq \frac{\log (3X)}{\log \psi}.
\end{equation}
Using \eqref{eqn10}, we have
\begin{equation}\label{eqn14}
    m_{1} < \frac{\log (33 X^2)}{\log \psi}.
\end{equation}
If \eqref{eqn12} not holds, then
\[
\frac{\gamma^{m_2}}{8X} \leq \frac{\gamma^{m_2}}{2X} - \frac{2X}{\gamma^{m_{2}}} \leq 8X \gamma^{m_{3}}.
\]
Then $\gamma^{m_{2} - m_{3}} \leq 64 X^2$, so
\begin{equation}\label{eqn15}
    m_{2} - m_{3} \leq 2 \frac{\log (8X)}{\log \psi}.
\end{equation}
Rewrite the \eqref{eqn8}, we have
\begin{align}\label{eqn16}
   &\left\rvert \gamma^{m_3} \left( A_{1} \gamma^{m_1 - m_3} + A_{2} \gamma^{m_2 - m_3} + A_3 \right) - \left( \frac{A_1}{\gamma^{m_1}} + \frac{A_2}{\gamma^{m_2}} + \frac{A_3}{\gamma^{m_3}} \right) \right\lvert \nonumber \\ 
   &~~~~~~~~~~~~~~~~~~~~~~~~~~= \left\rvert -A_{4} (\gamma^{m_4} - \delta^{m_4}) -A_{5} (\gamma^{m_5} - \delta^{m_5}) -A_{6} (\gamma^{m_6} - \delta^{m_6}) \right\lvert \nonumber \\
   &~~~~~~~~~~~~~~~~~~~~~~~~~~\leq 6X \gamma^{m_{4}}.
\end{align}
Assume $A_{1} \gamma^{m_1 - m_3} + A_{2} \gamma^{m_2 - m_3} + A_3 = 0$. Let $i = m_2 - m_3, j = m_1 - m_3$. Then
\[
j \leq \frac{\log (704 X^3)}{\log \psi} ~~~\text{and}~~~i \leq 2 \frac{\log (8X)}{\log \psi}
\]
are bounded. Therefore one can find all polynomials $A_{1}X^j + A_{2}X^i + A_3$ and took that polynomial which has a root $\gamma$, where $\gamma$ is a quadratic unit of norm $1$. For these balancing-like sequences,  $\delta$ is also a root of the same polynomial so that the left-hand side of \eqref{eqn16} is zero for any $m_3$. This shows that also $m_4 = m_5 = m_6 = 0$. Then we have

\[
(m_1, m_2, m_3, m_4, m_5, m_6) = (m_3 + j, m_3 + i, m_3, 0, 0, 0)
\]
is a parametric family of solutions. From now on assume that $A_{1} \gamma^{m_1 - m_3} + A_{2} \gamma^{m_2 - m_3} + A_3 \neq 0$. Thus 
$\left\rvert A_{1} \gamma^{m_1 - m_3} + A_{2} \gamma^{m_2 - m_3} + A_3 \right\lvert \left\rvert A_{1} \delta^{m_1 - m_3} + A_{2} \delta^{m_1 - m_3} + A_3 \right\lvert \geq 1.$ Since 
\[
\left\rvert A_{1} \delta^{m_1 - m_3} + A_{2} \delta^{m_1 - m_3} + A_3 \right\lvert \leq 3X,
\] 
which implies $\left\rvert A_{1} \gamma^{m_1 - m_3} + A_{2} \gamma^{m_2 - m_3} + A_3 \right\lvert \geq \frac{1}{3X}.$ Further,
\[
\left\rvert \frac{A_1}{\gamma^{m_1}} + \frac{A_2}{\gamma^{m_2}} + \frac{A_3}{\gamma^{m_3}} \right\lvert \leq \frac{3X}{\gamma^{m_3}}.
\]
Hence,
\[
\left\rvert \gamma^{m_3} \left( A_{1} \gamma^{m_1 - m_3} + A_{2} \gamma^{m_2 - m_3} + A_3 \right) - \left( \frac{A_1}{\gamma^{m_1}} + \frac{A_2}{\gamma^{m_2}} + \frac{A_3}{\gamma^{m_3}} \right) \right\lvert \geq \frac{\gamma^{m_3}}{3X} - \frac{3X}{\gamma^{m_3}}.
\]
If $\frac{\gamma^{m_3}}{3X} - \frac{3X}{\gamma^{m_3}} \leq \frac{\gamma^{m_3}}{6X}$, then $\gamma^{2m_3} \leq 18 X^2$, so $\gamma^{m_3} < 5X$. Hence,
\begin{equation}\label{****}
   m_{6} \leq m_{5} \leq m_{4} \leq m_3 \leq \frac{\log (5X)}{\log \psi}.
\end{equation}
Using \eqref{eqn10} and \eqref{eqn15}, we have
\begin{equation}\label{eq****}
    m_2 \leq \frac{\log (320 X^3)}{\log \psi}~~\text{and}~~m_1 \leq \frac{\log (3520 X^4)}{\log \psi}.
\end{equation}
Now, assume $\frac{\gamma^{m_3}}{6X} \leq \frac{\gamma^{m_3}}{3X} - \frac{3X}{\gamma^{m_3}} \leq 6X \gamma^{m_4}$, then $\gamma^{m_{3} - m_{4} } \leq 36 X^2$. Hence,
\begin{equation}\label{kk}
    m_3 - m_4 \leq 2 \frac{\log (6X)}{\log \psi}.
\end{equation}
Rewrite the \eqref{eqn8}, we have
\begin{align}\label{ll}
   &\left\rvert \gamma^{m_4} \left( A_{1} \gamma^{m_1 - m_4} + A_{2} \gamma^{m_2 - m_4} + A_3 \gamma^{m_3 - m_4} + A_4 \right) - \left( \frac{A_1}{\gamma^{m_1}} + \frac{A_2}{\gamma^{m_2}} + \frac{A_3}{\gamma^{m_3}} + \frac{A_4}{\gamma^{m_4}} \right) \right\lvert \nonumber \\ 
   &= \left\rvert -A_{5} (\gamma^{m_5} - \delta^{m_5}) -A_{6} (\gamma^{m_6} - \delta^{m_6}) \right\lvert \nonumber \\
   &\leq 4X \gamma^{m_{5}}.
\end{align}
Assume $A_{1} \gamma^{m_1 - m_4} + A_{2} \gamma^{m_2 - m_4} + A_3\gamma^{m_3 - m_4} = 0$. Let $i = m_3 - m_4, j = m_2 - m_4, k = m_1 - m_4$. Then
\[
k \leq \frac{\log (25,344 X^5)}{\log \psi},~~~~~ j \leq \frac{\log (2304 X^4)}{\log \psi} ~~~\text{and}~~~i \leq 2 \frac{\log (6X)}{\log \psi}
\]
are bounded. Therefore one can find all polynomials $A_{1}X^k + A_{2}X^j + A_{3}X^i + A_4$ and took that polynomial which has a root $\gamma$, where $\gamma$ is a quadratic unit of norm $1$. For these balancing-like sequences,  $\delta$ is also a root of the same polynomial so that the left-hand side of \eqref{ll} is zero for any $m_4$. This shows that also $m_5 = m_6 = 0$. Then we have

\[
(m_1, m_2, m_3, m_4, m_5, m_6) = (m_4 + k, m_4 + j, m_4 + i, m_4, 0, 0)
\]
is a parametric family of solutions. From now on assume that $A_{1} \gamma^{m_1 - m_4} + A_{2} \gamma^{m_2 - m_4} + A_3 \gamma^{m_3 - m_4} + A_4 \neq 0$. Thus 
\[
\left\rvert A_{1} \gamma^{m_1 - m_4} + A_{2} \gamma^{m_2 - m_4} + A_3 \gamma^{m_3 - m_4} + A_4 \right\lvert \left\rvert A_{1} \delta^{n_m - m_4} + A_{2} \delta^{m_2 - m_4} + A_3 \delta^{m_3 - m_4} + A_4 \right\lvert \geq 1.
\] 
Since $\left\rvert A_{1} \delta^{m_1 - m_4} + A_{2} \delta^{m_2 - m_4} + A_3 \delta^{m_3 - m_4} + A_4 \right\lvert \leq 4X$, which implies 
\[
\left\rvert A_{1} \gamma^{m_1 - m_4} + A_{2} \gamma^{m_2 - m_4} + A_3 \gamma^{m_3 - m_4} + A_4 \right\lvert \geq \frac{1}{4X}.
\]
Further,
\[
\left\rvert \frac{A_1}{\gamma^{m_1}} + \frac{A_2}{\gamma^{m_2}} + \frac{A_3}{\gamma^{m_3}} + \frac{A_4}{\gamma^{m_4}} \right\lvert \leq \frac{4X}{\gamma^{m_4}}.
\]
Hence,
\[
\left\rvert \gamma^{m_4} \left( A_{1} \gamma^{m_1 - m_4} + A_{2} \gamma^{m_2 - m_4} + A_3 \gamma^{m_3 - m_4} + A_4 \right) - \left( \frac{A_1}{\gamma^{m_1}} + \frac{A_2}{\gamma^{m_2}} + \frac{A_3}{\gamma^{n_3}} + \frac{A_4}{\gamma^{n_4}} \right) \right\lvert \geq \frac{\gamma^{m_4}}{4X} - \frac{4X}{\gamma^{m_4}}.
\]
If $\frac{\gamma^{m_4}}{4X} - \frac{4X}{\gamma^{m_4}} \leq \frac{\gamma^{m_4}}{8X}$, then $\gamma^{2m_4} \leq 32 X^2$, so $\gamma^{m_4} < 6X$. Hence,
\begin{equation}\label{eqn17}
  m_{6} \leq m_{5} \leq m_4 \leq \frac{\log (6X)}{\log \psi}.
\end{equation}
Using \eqref{eqn10}, \eqref{eqn15} and \eqref{kk}, we have
\begin{equation}\label{eq***}
    m_3 \leq \frac{3 \log (6 X)}{\log \psi},~~ m_2 \leq \frac{\log (13824 X^5)}{\log \psi}~~\text{and}~~m_1 \leq \frac{\log (152064 X^6)}{\log \psi}.
\end{equation}
Now, assume $\frac{\gamma^{m_4}}{8X} \leq \frac{\gamma^{m_4}}{4X} - \frac{4X}{\gamma^{m_4}} \leq 4X \gamma^{m_5}$, then $\gamma^{m_{4} - m_{5} } \leq 32 X^2$. Hence,
\begin{equation}\label{eqn18}
    m_4 - m_5 \leq \frac{\log (32 X^2)}{\log \psi}.
\end{equation}
Again, rewrite the equation \eqref{eqn8}, we have
\begin{align}\label{eqn19}
   & \Big| \gamma^{m_5} \left( A_{1} \gamma^{m_1 - m_5} + A_{2} \gamma^{m_2 - m_5} + A_3 \gamma^{m_3 - m_5} + A_4 \gamma^{m_4 - m_5} + A_5 \right) \nonumber \\
   &~~~~~~~~~~~~~~~~~~~~~~~~~~~~~~~~~~~~~~~~~~~~~~~~~~~~~~~~~~~~~~~~- \left( \frac{A_1}{\gamma^{m_1}} + \frac{A_2}{\gamma^{m_2}} + \frac{A_3}{\gamma^{m_3}} + \frac{A_4}{\gamma^{m_4}} + \frac{A_5}{\gamma^{m_5}} \right) \Big| \nonumber \\ 
 &= \left\rvert -A_{6} (\gamma^{m_6} - \delta^{m_6}) \right\lvert \nonumber\\
 & \leq 2X \gamma^{m_{6}}. 
\end{align}
Assume $A_{1} \gamma^{m_1 - m_5} + A_{2} \gamma^{m_2 - m_5} + A_3 \gamma^{m_3 - m_5} + A_4 \gamma^{m_4 - m_5} + A_5 = 0$. Let $i = m_4 - m_5, j = m_3 - m_5, k = m_2 - m_5, l = m_1 - m_5$. Then
\[
l \leq \frac{\log (811008 X^7)}{\log \psi},~~~ k \leq \frac{\log (73728 X^6)}{\log \psi},~~~ j \leq \frac{\log (1152 X^4)}{\log \psi} ~~~\text{and}~~~i \leq \frac{\log (32 X^2)}{\log \psi}
\]
are bounded. Therefore one can find all polynomials $A_{1}X^l + A_{2}X^k + A_{3}X^j + A_4X^i + A_5$ and took that polynomial which has a root $\gamma$, where $\gamma$ is a quadratic unit of norm $1$. For these balancing-like sequences, $\delta$ is also a root of the same polynomial so that the left-hand side of \eqref{eqn19} is zero for any $m_5$. This shows that also $ m_6 = 0$. Then we have

\[
(m_1, m_2, m_3, m_4, m_5, m_6) = (m_5 + l, m_5 + k, m_5 + j, m_5 + i, m_5, 0)
\]
is a parametric family of solutions. From now on assume that 
\[
A_{1} \gamma^{m_1 - m_5} + A_{2} \gamma^{m_2 - m_5} + A_3 \gamma^{m_3 - m_5} + A_4 \gamma^{m_4 - m_5} + A_5 \neq 0.
\]
Thus 
\begin{align*}
    &\left\rvert A_{1} \gamma^{m_1 - m_5} + A_{2} \gamma^{m_2 - m_5} + A_3 \gamma^{m_3 - m_5} + A_4 \gamma^{m_4 - m_5} + A_5 \right\lvert\\
    &~~~~~~~~~~~~~~~~~~~~~~~~~~~~~~~\times \left\rvert A_{1} \delta^{m_1 - m_5} + A_{2} \delta^{m_2 - m_5} + A_3 \delta^{m_3 - m_5} + A_4 \delta^{m_4 - m_5} + A_5 \right\lvert \geq 1.
\end{align*}
Since $\left\rvert A_{1} \delta^{m_1 - m_5} + A_{2} \delta^{m_2 - m_5} + A_3 \delta^{m_3 - m_5} + A_4 \delta^{m_4 - m_5} + A_5 \right\lvert \leq 5X$, which implies 
\[
\left\rvert A_{1} \gamma^{m_1 - m_5} + A_{2} \gamma^{m_2 - m_5} + A_3 \gamma^{m_3 - m_5} + A_4 \gamma^{m_4 - m_5} + A_5 \right\lvert \geq \frac{1}{5X}.
\]
Further,
\[
\left\rvert \frac{A_1}{\gamma^{m_1}} + \frac{A_2}{\gamma^{m_2}} + \frac{A_3}{\gamma^{m_3}} + \frac{A_4}{\gamma^{m_4}} + \frac{A_5}{\gamma^{m_5}} \right\lvert \leq \frac{5X}{\gamma^{m_5}}.
\]
Hence, the left-hand side of \eqref{eqn19} is at least as large as
\[
\frac{\gamma^{m_5}}{5X} - \frac{5X}{\gamma^{m_5}}.
\]
If $\frac{\gamma^{m_5}}{5X} - \frac{5X}{\gamma^{m_5}} \leq \frac{\gamma^{m_5}}{10X}$, then $\gamma^{2m_5} \leq 50 X^2$, so $\gamma^{m_5} < 8X$. Hence,
\begin{equation}\label{eqn20}
  m_{6} \leq m_5 \leq \frac{\log (8X)}{\log \psi}.
\end{equation}
Using \eqref{eqn10}, \eqref{eqn15}, \eqref{kk} and \eqref{eqn18}, we have
\begin{equation}\label{eq**}
    m_4 \leq \frac{\log (256 X^3)}{\log \psi},~~ m_3 \leq \frac{\log (9216 X^5)}{\log \psi},~~ m_2 \leq \frac{\log (589824 X^7)}{\log \psi}~~\text{and}~~m_1 \leq \frac{\log (6488064 X^8)}{\log \psi}.
\end{equation}
Now, assume $\frac{\gamma^{m_5}}{10X} \leq \frac{\gamma^{m_5}}{5X} - \frac{5X}{\gamma^{m_5}} \leq 2X \gamma^{m_6}$, then $\gamma^{m_{5} - m_{6} } \leq 20 X^2$. Hence,
\begin{equation}\label{eqn21}
    m_5 - m_6 \leq \frac{\log (20 X^2)}{\log \psi}.
\end{equation}
Again, rewrite the equation \eqref{eqn8}, we have
\begin{align}\label{eqn22}
  &\left\rvert \gamma^{m_6} \left( A_{1} \gamma^{m_1 - m_6} + A_{2} \gamma^{m_2 - m_6} + A_3 \gamma^{m_3 - m_6} + A_4 \gamma^{m_4 - m_6} + A_5 \gamma^{m_5 - m_6} + A_6 \right) \right\rvert \nonumber \\
  &~~~~~~~~~~~~= \left\rvert \delta^{m_6} \left( A_{1} \delta^{m_1 - m_6} + A_{2} \delta^{m_2 - m_6} + A_3 \delta^{m_3 - m_6} + A_4 \delta^{m_4 - m_6} + A_5 \delta^{m_5 - m_6} + A_6 \right) \right\rvert
\end{align}
Assume left-hand side of \eqref{eqn22} is zero. Let $i = m_5 - m_6, j = m_4 - m_6, k = m_3 - m_6, l = m_2 - m_6, m = m_1 - m_6$. Then
\begin{align*}
    m & \leq \frac{\log (16220160 X^9)}{\log \psi},~~~ l \leq \frac{\log (1474560 X^8)}{\log \psi},~~~ k \leq \frac{\log (23040 X^6)}{\log \psi},\\
    &j \leq \frac{\log (640 X^4)}{\log \psi} ~~~\text{and}~~~i \leq \frac{\log (20 X^2)}{\log \psi}
\end{align*}
are bounded. Therefore one can find all polynomials $A_{1}X^m + A_{2}X^l + A_{3}X^k + A_{4}X^j + A_5 X^i + A_6$ and took that polynomial which has a root $\gamma$, where $\gamma$ is a quadratic unit of norm $1$.  For such, \eqref{eqn22} holds for all $m_6$. Hence, we have the parametric family of solutions

\[
(m_1, m_2, m_3, m_4, m_5, m_6) = (m_6 + m, m_6 + l, m_6 + k, m_6 + j, m_6 + i, m_6).
\]
From now on assume that left-hand side of \eqref{eqn22} is non-zero. Thus 
\begin{align*}
    &\left\rvert A_{1} \gamma^{m_1 - m_6} + A_{2} \gamma^{m_2 - m_6} + A_3 \gamma^{m_3 - m_6} + A_4 \gamma^{m_4 - m_6} + A_5 \gamma^{m_5 - m_6} + A_6 \right\lvert\\
    &~~~~~~~~~~~~~~~~~~~~~~~\times \left\rvert A_{1} \delta^{m_1 - m_6} + A_{2} \delta^{m_2 - m_6} + A_3 \delta^{m_3 - m_6} + A_4 \delta^{m_4 - m_6} + A_5 \delta^{m_5 - m_6} + A_6 \right\lvert \geq 1.
\end{align*}
Left–hand side second factor of the above equation is $\leq 6 X$. Hence
\[
\left\rvert A_{1} \gamma^{m_1 - m_6} + A_{2} \gamma^{m_2 - m_6} + A_3 \gamma^{m_3 - m_6} + A_4 \gamma^{m_4 - m_6} + A_5 \gamma^{m_5 - m_6} + A_6 \right\lvert \geq \frac{1}{6X}.
\]
Hence, equation \eqref{eqn22} becomes
\[
\frac{\gamma^{m_6}}{6X} \leq 6 \delta^{m_6} = \frac{6X}{\gamma^{m_6}},
\]
which gives 
\begin{equation}\label{eqn23}
  m_6 \leq \frac{\log (6X)}{\log \psi}.
\end{equation}
Using \eqref{eqn10}, \eqref{eqn15}, \eqref{kk}, \eqref{eqn18} and \eqref{eqn21}, we have
\begin{align}\label{eq*}
&m_6 \leq \frac{\log (6 X)}{\log \psi}, \nonumber \\
&m_5 \leq \frac{\log (120 X^3)}{\log \psi}, \nonumber \\
 &m_4 \leq \frac{\log (3840 X^5)}{\log \psi}, \nonumber \\
 &m_3 \leq \frac{\log (138240 X^7)}{\log \psi},  \\
 &m_2 \leq \frac{\log (8847360 X^9)}{\log \psi}, \nonumber \\
 &m_1 \leq \frac{\log (97320960 X^{10})}{\log \psi}. \nonumber
\end{align}
Note that \eqref{eq*} contains \eqref{eqn20}, \eqref{eq**}, \eqref{eqn17}, \eqref{eq***}, \eqref{****}, \eqref{eq****}, \eqref{eqn13} and \eqref{eqn14}. Recalling that we have to replace $X$ by $2X$, this completes the proof of the Theorem \eqref{th:1}.

\begin{center}
    \section{Numerical example}
\end{center}
Take $A_1, A_2, A_3, A_4, A_5, A_6 \in \{0,\pm 1\}$. Hence $X = 1$, therefore $A \leq 308$. Thus, Theorem \ref{th:1} says that the sporadic solutions are of the form
\[
x_{m_{1}} \pm A_2 x_{m_{2}} \pm A_3 x_{m_{3}} \pm A_4 x_{m_{4}} \pm A_5 x_{m_{5}}\pm A_6 x_{m_{6}}=0,
\]
where
$A_2, A_3, A_4, A_5, A_6 \in \{0,\pm 1\}$, $m_1 > m_2 \geq m_3 \geq m_4 \geq m_5 \geq m_6 \geq 0$. Here $m_6 < 3, m_5 < 8, m_4 < 13, m_3 < 18, m_2 < 24$ and $m_1 > m_2$. Now for $r \in [3, 308], m_6 \in [0, 3], m_5 \in [m_6, 8], m_4 \in [m_5, 13], m_3 \in [m_4, 24], m_2 \in [m_3, 24], \alpha_{1} \in \{0,1\}$ and  $\alpha_{2}, \alpha_{3}, \alpha_{4}, \alpha_{5} \in \{0, \pm 1\}$, we search the equation
\[
x_{m_{1}} = \mid \alpha_{2} x_{m_{2}} + \alpha_{3} x_{m_{3}} + \alpha_{4} x_{m_{4}} + \alpha_{5} x_{m_{5}} + \alpha_{6} x_{m_{6}} \mid,
\]
holds for some $m_1 > m_2$. Using {\it Mathematica} software and find only one solution when $A = 5$, that is,
\[
x_2 = x_1 + x_1 + x_1 + x_1 + x_1. 
\]

\begin{flushright}
    \begin{tabular}{l}
    \textit{Department of Mathematics}\\
    \textit{KIIT Deemed to be University, Bhubaneswar, India}\\
\textit{iiit.bijan@gmail.com; bijan.patelfma@kiit.ac.in} \\
    \end{tabular}
\end{flushright} 

\begin{flushright}
    \begin{tabular}{l}
    \textit{Department of Mathematics}\\
    \textit{IISER Bhopal, Madhya Pradesh, India}\\
\textit{prashanttiwaridav@gmail.com; ptiwari@iiserb.ac.in} \\
    \end{tabular}
\end{flushright}
\end{document}